\documentclass{article}

\usepackage[marginparwidth=3cm, marginparsep=.5cm]{geometry}
\usepackage{authblk}

\usepackage{graphicx} % Required for inserting images
\usepackage[utf8]{inputenc}
\usepackage{amsmath,amsthm,amsfonts,amssymb,bm, mathdots}
\usepackage{authblk}
\usepackage{tikz}
\usepackage{ytableau}

\usepackage{mathtools}
\mathtoolsset{showonlyrefs}

\usepackage{caption}
\usepackage{subcaption}
\usepackage{enumerate}
\usepackage{hyperref}

\usepackage{soul}

\theoremstyle{plain}
\newtheorem{thm}{Theorem}[section]
\newtheorem{theorem}{Theorem}[section]
\newtheorem*{thm*}{Theorem}
\newtheorem{prop}[thm]{Proposition}
\newtheorem*{prop*}{Proposition}

\newtheorem*{conj*}{Conjecture}
\newtheorem{cor}[thm]{Corollary}
\newtheorem{lem}[thm]{Lemma}
\newtheorem*{lem*}{Lemma}

\renewcommand{\i}{\mathrm{i}}

\renewcommand{\Re}{\operatorname{Re}}
\renewcommand{\Im}{\operatorname{Im}}

\title{Temperleyan Domino Tilings with Holes}

\author{Matthew Nicoletti \thanks{University of California, Berkeley, Department of Statistics\\
 E-mail: mnicoletti@berkeley.edu}}
\date{}

\begin{document}

\maketitle

\begin{abstract}
    We analyze asymptotic height function fluctuations in uniformly random domino tiling models on multiply connected 
    Temperleyan domains. Starting from asymptotic formulas derived by Kenyon \cite{Ken99}, we show that (1) the difference of the centered height function and a harmonic function with boundary values given by the (random) centered hole heights converges in the sense of moments to a Gaussian free field, which is independent of the hole heights, and (2) the hole heights themselves converge in distribution to a discrete Gaussian random vector. These results confirm general predictions about height fluctuations for tilings on multiply connected domains.
\end{abstract}

% \tableofcontents

\section{Introduction}

% We consider tilings of domains of the form illustrated in Figure~\ref{fig:one_hole_ex}. 

% \begin{figure}
%     \centering
%     \includegraphics[width=0.5\linewidth]{pictures/temperleyan_example.png}
%     \caption{There is a single ``marked square'' sticking out into the interior of each hole.}
%     \label{fig:one_hole_ex}
% \end{figure}

\subsection{Overview}
The dimer model, the study of random perfect matchings on bipartite graphs, or equivalently of random tilings of domains in the plane is a well studied model in statistical mechanics which exhibits a wide array of universal behaviors. Using \emph{Thurston's height function}~\cite{Thu90}, conformal invariance of the scaling limit of the model has been established in many settings, providing rigorous proofs of general physical predictions. We refer to the surveys~\cite{Gor21,Ken04,Ken09} for more general history and background.

In this note, we study uniformly random domino tilings of \emph{Temperleyan domains}, analyzed by Kenyon \cite{Ken99}. These are rectilinear regions approximating a fixed (possibly multiply connected) region~$U \subset \mathbb{C}$. As explained in that work, the name comes from a bijection of Temperley~\cite{Tem81} generalized in~\cite{KPW00}. In~\cite{Ken99}, the moments of the height function are shown to have a conformally invariant limit. The proof of this result is constructive; an expression for joint moments as an iterated contour integral  is computed (by analyzing the inverse Kasteleyn using discrete complex analysis techniques), and the contour integral is shown to be invariant under conformal isomorphisms. For general~$U$, the integrand in the formula for a joint height moment is not explicit, though if~$U$ is simply connected, it is computed exactly. In the seminal follow up work~\cite{Ken01}, for simply connected~$U$, the integral formulas are identified with the moments of a Gaussian free field.

Our main result is the characterization of the scaling limit of the height fluctuations in the Temperleyan setting for multiply connected domains~$U$ with piecewise smooth boundary. Asymptotic joint moment formulas of~\cite{Ken99} are our starting point, and our approach is to use theta functions on an associated Riemann surface~$R$ (the double of~$U$) to analyze those expressions. In particular, Theorems~\ref{thm:intro_main1} and~\ref{thm:intro_main2} below identify the (limiting) moments with those of the independent sum of a Gaussian free field on~$U$ and a harmonic function whose boundary values on each of the inner boundary components of~$U$ is a random constant; moreover, these constants are jointly distributed as a centered discrete Gaussian distribution. The result confirms general predictions for multiply connected tiling models given in~\cite[Conjecture 24.2]{Gor21}.

The discrete Gaussian arises because boundary heights are not fixed along inner boundaries. This may be compared to tiling models where heights along all boundaries are fixed; in such a setting, convergence to a Gaussian free field without any additional discrete component has been shown~\cite{BG19}. Since local height differences are deterministic along boundary components, mean-subtracted hole-boundary height values are well defined independently of the choices of representative boundary lattice sites. In our setting, mean-subtracted hole-boundary heights (up to a factor of~$\frac{1}{4}$ due to the height function convention) converge in distribution to a multivariate discrete Gaussian distribution, whose components are the boundary values of the random harmonic function described above.

Our computation involves the identification of the \emph{shift parameter} in the discrete Gaussian distribution (denoted as~$e$ in Theorem~\ref{thm:intro_main2}) for the class of domains we consider; a general formula for this parameter remains unknown even at the level of heuristics, see the discussion surrounding the conjectures in Section 24.2 of~\cite{Gor21}, where this parameter is called~$m$. In both the present work and in~\cite{BN25} (discussed more below), the shift in the discrete Gaussian comes naturally from a \emph{standard divisor} of a theta function on an associated compact Riemann surface. However, in contrast to that work and other works involving discrete Gaussians (discussed below), here the shift does not evolve quasi-periodically as a function of the lattice scale parameter~$\epsilon$; here the corresponding standard divisor is fixed and simply consists of the collection~$d_1,\dots, d_g$ of marked points on inner boundaries of~$U$.

%As far as we are aware, the results in this work provide the first explicit characterization of dimer model height fluctuations in a generic class of multiply connected domains (without conditioning on heights of inner boundaries). 

The essential new insight is Lemma~\ref{lem:maincomp}, which uses theta functions on an associated compact Riemann surface to ``explicitly'' compute the integrand in the integral formulas derived in the work of Kenyon. This computation appears to be new, even though the moment formulas of~\cite{Ken99} have been known for many years. The other essential inputs are the arguments of~\cite[Section 4]{BN25}; that work derives integral formulas for joint height moments of the same exact form as in Corollary~\ref{cor:moments2} here, and from them the Gaussian free field and discrete Gaussian components are extracted in a general way.

 Our results provide another indication (in addition to the various results involving discrete Gaussians discussed below) that discrete Gaussians are universal in multiply connected 2D statistical mechanics models. Moreover, we conjecture that for a very large class of ``higher genus'' dimer models, even the integral formula for height moments, of the form given in Corollary~\ref{cor:moments2} (compare also with~\cite[Lemma 4.4]{BN25}, which leads to moment formulas with the same structure), is universal. Indeed, the genus zero version of the formula, which amounts to taking~$\omega_0(z,z') = \frac{dz}{z-z'}$ there, appears to be universal in simply connected and genus zero models, as has been confirmed in many large classes of examples, such as~\cite{Ken01,Ken08,BF14,Dui13,Pet15}. It may be particularly interesting to note that the result we obtain here, for domains with holes, matches the results obtained in~\cite{BN25}, which analyzes the Aztec diamond setup with gaseous facets emerging in the bulk. The only differences are the metric underlying the Gaussian free field and the parameters of the discrete Gaussian distributions.

 There are many proofs of Gaussian fluctuations in \emph{simply connected} tiling models. Convergence with more general ``flat'' boundary conditions using discrete complex analysis was obtained in~\cite{Rus18,Rus20}; moreover, much more general discrete complex analysis techniques for dimer models were developed in~\cite{CLR1,CLR2}, and were further analyzed in special cases in~\cite{BNR23, BNR24}. Many works prove convergence in other setups using a variety of tools, including~\cite{Ken08,BF14,Dui13,Pet15,BK17,BG18,BL18,Hua20,GH22}. For non simply connected models, aside from the works~\cite{BG19} and~\cite{BN25} discussed above, there are fewer results. An exact calculation of the asymptotic distribution of the number of nontrivial loops of a double dimer model on a cylinder appears in \cite{Ken14}, and there it is also shown that the double-dimer loops in multiply connected domains are conformally invariant. A tiling model on a cylinder was studied in~\cite{ARV21}, and a result exactly analogous to our Theorems~\ref{thm:intro_main1} and~\ref{thm:intro_main2} below is obtained; this work provided the first computation of fluctuations in a non simply connected setup. One difference in that setup is that the discrete Gaussian distribution is present in the model from the outset, and moreover the approach (cleverly, using a different underlying integrable structure) bypasses a direct analysis of the correlation kernel, which appears to be the only possible route in the present setup.

Our results may be further compared to a variety of other results involving discrete Gaussians in the literature on random point processes, and in particular on random matrix models and dimer models on surfaces. In the context of random point processes related to random matrix ensembles, theta functions appeared in the asymptotic expansions of certain large deviations events for the sine kernel process~\cite{DIZ97}. Theta functions and discrete Gaussians appear in the physics papers~\cite{BDE00},~\cite{Eyn09}, as well as in the mathematical works~\cite{Shc13,BG24}; these works all analyze \emph{$\beta$ ensembles} (which generalize random Hermitian matrix models) in the multi-cut regime. Additional examples of discrete Gaussians describing asymptotic behaviors in statistical mechanics models include~\cite{ACC22, ACCL24,Cha24} which analyze certain 2D Coulomb gas models in multiply connected regimes. See also references within these works.

Dimer models on various discretizations of a torus were studied in~\cite{BT06},~\cite{Dub15},~\cite{DG15},~\cite{KSW16}. The works~\cite{BT06} and \cite{KSW16} show that a discrete Gaussian describes the random monodromies of the multivalued height function on the torus;~\cite{Dub15} and ~\cite{DG15} obtain a decomposition of the height fluctuations as a Gaussian free field plus an independent discrete Gaussian times a harmonic function. As they explain, this object is also known as the \emph{compactified free field}. 

There has been recent work studying dimer models on higher genus surfaces as well. The sequence of works \cite{BLR24,BLR25} analyze dimer models on certain Temperleyan graphs embedded in Riemann surfaces (with arbitrary genus and possibly boundary components). To make the graphs have perfect matchings, they remove a certain number of white vertices from the approximating graphs, and in the limit these removed white vertices converge to marked points on the surface. Under certain natural assumptions on the sequence of graphs, those works prove convergence to a universal limit, invariant under conformal transformations of the surface with marked points; they do not characterize the limit, though they conjecture that it is a compactified free field. By further developing and applying the machinery of \emph{t-embeddings}, together with the technique of computing a family of perturbed Kasteleyn determinants in order to access observables,~\cite{Bas24} identifies this limit in a collection of cases which includes all isomorphism classes of limiting Riemann surfaces with marked points. Modulo the verification of a technical condition (which is expected to be true and will be verified in future work) required for the universality theorems of \cite{BLR24,BLR25} to be applicable, this identification completes the picture and proves convergence to a compactified free field for a very large family of dimer models on surfaces.

In particular, after the completion of a first version of this work, the author learned that the setup of~\cite{Bas24} contains multiply connected planar domains (the subject of this work) as a special case. However, aside from the technical condition mentioned above, there is one other reason that our results do not follow directly from the combination of~\cite{Bas24} and \cite{BLR24, BLR25}: In those works, to balance the number of black and white vertices a certain number of \emph{interior} white vertices are removed from the graph, whereas in our setting (following~\cite{Ken99}) we add certain \emph{boundary} black vertices to make the domain tileable. Thus, roughly speaking, our setup should correspond to a limiting case of theirs where marked points merge in pairs at the boundary components. Moreover, our methods are quite different, as we proceed by directly analyzing moments via the inverse of the (unperturbed) Kasteleyn matrix, and we thereby make a connection to an analogous result for the Aztec diamond (via~\cite{BN25}, as discussed above).

\subsection{Results}

Consider a checkerboard coloring of unit lattice squares tiling~$\mathbb{R}^2$, with each square centered at a point of~$\mathbb{Z}^2$ and the square centered at~$(0,0)$ colored white. Let~$W_0$, resp.~$W_1$, be the set of unit squares with both coordinates even, resp. odd. Let~$B_0$, resp. $B_1$, denote the set of unit squares with coordinates equal to~$(1,0)$ mod $2$, resp. $(0,1)$ mod $2$.

An \emph{even polyomino} is a union of lattice squares bounded by simple closed lattice paths, such that all corner squares (at convex or concave corners) are of type~$B_1$. A \emph{Temperleyan  polyomino} is an even polyomino with a black square~$\tilde d_0$ on the outer boundary removed, and with one black square~$\tilde d_j$ added along along each inner boundary component.  A \emph{domino tiling} of a Temperleyan polyomino is a tiling of it by $2 \times 1$ rectangles consisting of pairs of adjacent lattice squares. A Temperleyan polyomino on~$\epsilon \mathbb{Z}^2$ is a Temperleyan polyomino rescaled by~$\epsilon$, so the corresponding rescaled dominoes are~$2 \epsilon \times \epsilon$ rectangles. We will study the uniform measure on domino tilings of Temperleyan polyominos on~$\epsilon \mathbb{Z}^2$ approximating a fixed domain~$U$.

Suppose~$U$ is a connected domain with~$g+1$ piecewise smooth boundary components~$A_0,\dots,A_g$ and~$g+1$ marked points~$d_j$,~$j=0,\dots,g$, one along each boundary component. We assume, as in~\cite{Ken99}, that tangents along the boundary have one sided limits at corners. Let~$P_\epsilon$ be a Temperleyan polyomino on $\epsilon \mathbb{Z}^2$ approximating~$U$ in the following sense. The boundary components of~$P_\epsilon$ are within~$O(\epsilon)$ of those of~$U$, and away from corners of~$\partial U$, the tangent vector of $\partial U$ points in the same half space as the tangent at nearby points of the polyomino. Moreover, the removed vertex and exposed vertices~$\tilde d_j$ of~$P_\epsilon$ are within~$O(\epsilon)$ of~$d_j$,~$j=0,\dots,g$. Suppose in addition that in a~$\delta$ neighborhood of each~$\tilde d_j$ the boundary of~$P_\epsilon$ is flat (vertical or horizontal), where~$\delta = \delta(\epsilon)$ tends to zero sufficiently slowly (as required in the proof of~\cite[Theorem 13]{Ken99}).

The \emph{height function} of a domino tiling of a polyomino on~$\epsilon \mathbb{Z}^2$ is the function on vertices of the polyomino defined by declaring an outer boundary vertex~$v$ to have~$h(v) = 0$ together with the following local rules: For~$v$ and~$v + \epsilon$ adjacent lattice points of the polyomino such that the directed edge~$(v, v+\epsilon)$ has a white square on its left,~$h(v + \epsilon) = h(v) + 3$ if the directed edge crosses a domino, and otherwise~$h(v + \epsilon) = h(v) - 1$.

Our characterization of the height function will involve the \emph{Green's function} on~$U$, via the moments of the \emph{Gaussian free field} (see~\cite{She07} for a detailed definition and exposition), which appear in the right hand side of~\eqref{eqn:joint_moment_in} below. Let~$g_U(z,z')$ denote the Green's function on~$U$. The Green's function can be characterized as the unique function which for fixed~$z'$ is harmonic in~$z$ for~$z \neq z'$, and with the property that~$g_U(z,z')- (-\frac{1}{2\pi} \log|z-z'|) $ is smooth and harmonic for~$z$ in a neighborhood of~$z'$. It satisfies the symmetry property~$g_U(z,z') = g_U(z',z)$. In addition, let~$f_j: U \rightarrow \mathbb{R}$,~$j=1,\dots,g$ be the unique harmonic function satisfying
\begin{equation}\label{eqn:fjdef}
    f_j|_{\partial U}(z) = \begin{cases}
        1, & z \in A_j \\
        0, & z \in U \setminus A_j.
    \end{cases}
\end{equation}

Define~$h = h_\epsilon$ as the random height function corresponding to a uniformly random domino tiling of~$P_\epsilon$. Let~$Z_j = h(d_j^{(\epsilon)}) - \mathbb{E}[h(d_j^{(\epsilon)})]$ for boundary lattice points~$d_j^{(\epsilon)}$ along the boundary component of~$P_\epsilon$ corresponding to the component of~$\partial U$ containing~$d_j$. By~\cite[Proposition 20]{Ken99}, joint moments of $(Z_1,\dots,Z_g)$ converge as~$\epsilon \rightarrow 0$. For any lattice point~$v$ of~$P_{\epsilon}$, define
\begin{equation}\label{eqn:tildeh}
    \tilde{h}(v) \coloneqq h(v)-\mathbb{E}[h(v)] - \sum_{j=1}^g Z_j f_j(v).
\end{equation}
Our first theorem is the following.

\begin{theorem}\label{thm:intro_main1}
    Fix pairwise distinct points~$z_1,\dots, z_K \in U$. For each~$\epsilon$ choose lattice points~$z_1^{(\epsilon)},\dots, z_K^{(\epsilon)}$ of~$P_\epsilon$, such that~$z_i^{(\epsilon)}$ is within~$O(\epsilon) $ of~$z_i$. The joint moments of~$\tilde{h}$ converges to those of the Gaussian free field in~$U$: 
\begin{equation}\label{eqn:joint_moment_in}
     \lim_{\epsilon \rightarrow 0} \mathbb{E}\left[  \prod_{j=1}^K \tilde h(z_j^{(\epsilon)})\right] 
     =\begin{cases} 
     \frac{4^K}{\pi^{\frac{K}{2}}} \sum_{\pi = \{\{i,j\} \}} \prod_{\{i,j\} \in \pi} g_{U} (z_i, z_j), 
     & K \text{ even}\\
    0, &  K \text{ odd}.
    \end{cases}
    \end{equation}
The summation in~\eqref{eqn:joint_moment_in} is over all \emph{pairings}, or partitions of $\{1,\dots, K\}$ into subsets of size~$2$. Moreover, for any integers~$n_1, n_2,\dots, n_g \geq 0$,
\begin{equation}
    \lim_{\epsilon \rightarrow 0} \mathbb{E}\left[ Z_1^{n_1} \cdots Z_g^{n_g}  \prod_{j=1}^K \tilde h(z_j^{(\epsilon)}) \right] =   \lim_{\epsilon \rightarrow 0} \mathbb{E}\left[ Z_1^{n_1} \cdots Z_g^{n_g} \right] \lim_{\epsilon \rightarrow 0}   \mathbb{E}\left[\prod_{j=1}^K \tilde h(z_j^{(\epsilon)}) \right].
\end{equation}
In other words,~$\tilde{h}$ and~$(Z_1,\dots,Z_g)$ are asymptotically independent (in the sense of moments) as~$\epsilon \rightarrow 0$.
    
\end{theorem}

In addition, we have the following theorem explicitly characterizing the limit in distribution of~$(Z_1,\dots,Z_g)$. With~$f_1,\dots,f_g$ as above, define
\begin{equation}\label{eqn:taudef}
    \tau_{i j} \coloneqq \frac{1}{2} \int_{U} \nabla f_i \cdot \nabla f_j dx dy
\end{equation}
for $i,j=1,\dots,g$; this is a symmetric and positive definite~$g \times g$ matrix. As described in Section~\ref{subsec:doubleU}, the double of~$U$, obtained by gluing~$U$ to itself along its boundary, is a compact Riemann surface. In fact it is a special type of surface called an~$M$ curve~\cite{BCT22}. We also need
\begin{equation}\label{eqn:edefint}
e \coloneqq -\sum_{j=1}^g \int_{d_0}^{d_j} \vec{\omega} + \Delta,
\end{equation}
 where~$\vec{\omega} = (\omega_1,\dots,\omega_g)$ are the holomorphic one forms on the double of~$U$, and~$\Delta$ is the \emph{vector of Riemann constants}. The integration paths are taken to remain inside~$U$. By properties of~$M$ curves,~$e \in \mathbb{R}^g$. See Section~\ref{subsec:doubleU} for definitions and slightly more discussion.

\begin{theorem}\label{thm:intro_main2}
    Let~$(X_1,\dots,X_g)$ have the distribution supported on~$\mathbb{Z}^g$ and given by
\begin{equation}\label{eqn:disc_gauss}
     \mathbb{P}(X = n) = \frac{1}{C} \exp(- \pi (n-e) \cdot \tau (n-e)) \qquad n \in \mathbb{Z}^g
    \end{equation}
where~$C $ is a normalization constant, and~$\tau$ and~$e$ are defined in~\eqref{eqn:taudef} and~\eqref{eqn:edefint}, respectively.
Then we have the convergence in distribution as~$\epsilon \rightarrow 0$,
\begin{equation}\label{eqn:disc_gauss_conv_intro}
    (\frac{1}{4}Z_1,\dots,\frac{1}{4} Z_g) \stackrel{d}{\rightarrow}   (X_1-\mathbb{E}[X_1],\dots,X_g -\mathbb{E}[X_g]).
\end{equation}
    
\end{theorem}

 The probability distribution defined by the right hand side of~\eqref{eqn:disc_gauss} is called a \emph{discrete Gaussian distribution}. It is a Gaussian random vector conditioned to take values in~$\mathbb{Z}^g$; it is also Shannon entropy maximizing among probability distributions supported on~$\mathbb{Z}^g$ which have a fixed mean and covariance matrix~\cite{AA19}. The parameter~$e \in \mathbb{R}^g$ is called the \emph{shift parameter}, and~$\tau \in \mathbb{R}^{g \times g}$ is called the \emph{scale matrix}. Lecture 24.2 of~\cite{Gor21} predicts the form of the distribution of height fluctuations for random tilings of multiply connected regions, and in addition Conjecture 24.2 predicts the scale matrix of the discrete Gaussian. Theorems \ref{thm:intro_main1} and \ref{thm:intro_main2} confirm these predictions. The factor of~$\frac{1}{4}$ is due to the chosen height function convention.

The rest of the paper is organized as follows. Section~\ref{sec:prelim} states results from~\cite{Ken99}, and provides the necessary facts about Riemann surfaces and their associated theta functions that we need for this work. Then, Section~\ref{sec:joint_moments} provides an explicit computation in terms of theta functions of the formulas for moments given by~\cite{Ken99}, leading ultimately to the proofs of Theorems~\ref{thm:intro_main1} and~\ref{thm:intro_main2}.
\subsection{Acknowledgments}
The author thanks Vadim Gorin, Alexei Borodin, and Tomas Berggren for valuable feedback. The author was supported by the NSF grant No. DMS 2402237.

\section{Preliminaries}
\label{sec:prelim}
\subsection{Results of~\cite{Ken99}, functions~$F_+$ and~$F_-$}
\label{subsec:kenres}
Let~$U$ be a Jordan domain with~$g+1$ smooth boundary curves. The function~$F_0(z_1,z_2)$ on~$U \times U$ is uniquely defined by the following properties, viewed as a function of~$z_2$ with~$z_1$ fixed:
\begin{enumerate}
\item It is meromorphic.
\item It has zero real part
along~$\partial U$, except possibly at points~$d_1,\dots,d_g$ (where it may have a pole).
\item It has a simple pole at~$z_2 = z_1$ with residue~$1/\pi$, and may have at most a simple pole at $d_1,\dots,d_g$, and there are no other poles. 
    
\item It is zero at~$d_0$.
\end{enumerate}
    
Moreover, the function~$F_1(z_1,z_2)$ on~$U \times U$ is uniquely defined from the following properties, viewed as a function of~$z_2$ with~$z_1$ fixed:
\begin{enumerate}
\item It is meromorphic.
\item It has zero imaginary part
along~$\partial U$, except possibly at points~$d_1,\dots,d_g$ (where it may have a pole).
\item It has a simple pole at~$z_2 = z_1$ with residue~$1/\pi$, and may have at most a simple pole at $d_1,\dots,d_g$, and there are no other poles. 
\item It is zero at~$d_0$.
\end{enumerate}
As we will see, (and as is already implicit in~\cite{Ken99}) one can extend~$F_0$ and~$F_1$ to the double of~$U$, and this will be useful for explicit computations. 

The functions we need in the statement below are
\begin{align}\label{eqn:fpdef}
    F_+(z_1,z_2) &\coloneqq F_0(z_1,z_2) + F_1(z_1,z_2) \\
     F_-(z_1,z_2) &\coloneqq F_0(z_1,z_2) - F_1(z_1,z_2) .
     \label{eqn:fmdef}
\end{align}
Proposition 20 of~\cite{Ken99} states the following. The proposition below remains valid if some point~$z_i$ is on the boundary of~$U$.
\begin{prop}\label{prop:moments}
   Fix pairwise distinct points~$z_1,\dots, z_K \in U$, and for each~$\epsilon$ choose lattice points~$z_1^{(\epsilon)},\dots, z_K^{(\epsilon)}$ of~$P_\epsilon$, such that~$z_i^{(\epsilon)}$ is within~$O(\epsilon) $ of~$z_i$. Let~$\gamma_i$,~$i=1,\dots,K$ be paths joining~$A_0$ to~$z_i$, which are disjoint. The centered moment of heights converges
   \begin{multline}\label{eqn:joint_moment}
     \lim_{\epsilon \rightarrow 0} \mathbb{E}\left[  \prod_{j=1}^K (h(z_j^{(\epsilon)}) - \mathbb{E} [h(z_j^{(\epsilon)})]) \right]  \\
     = (-\i)^K\sum_{\varepsilon_1,\dots,\varepsilon_K \in \{\pm\}}  \varepsilon_1 \cdots \varepsilon_K \int_{\gamma_1} \cdots \int_{\gamma_K} \det(F_{\varepsilon_i, \varepsilon_j}(z_i^{(\varepsilon_i)}, z_j^{(\varepsilon_j)}))_{i,j=1}^K d z_1^{(\varepsilon_1)} \cdots d z_K^{(\varepsilon_K)}
   \end{multline}
   where~$d z_j^{(1)} = d z_j$ and~$d z_j^{(-1)}=d \bar z_j$, and
   \begin{equation}\label{eqn:Fepsdef}
       F_{\varepsilon_i, \varepsilon_j}(z_i, z_j) = \begin{cases}
           0 , &  i = j \\
           F_+(z_i, z_j) , &  (\varepsilon_i, \varepsilon_j) = (1,1) \\
           F_-(z_i, z_j) , &  (\varepsilon_i, \varepsilon_j) = (-1,1) \\
           \overline{F_-(z_i, z_j)} , &  (\varepsilon_i, \varepsilon_j) = (1,-1) \\
           \overline{F_+(z_i, z_j)} , &  (\varepsilon_i, \varepsilon_j) = (-1,-1).
       \end{cases}
   \end{equation}
\end{prop}
\begin{proof}
The proposition in~\cite{Ken99} is only stated for the case when each point~$z_i$ is on a boundary component of~$U$. However, as noted there, the proof clearly provides~\eqref{eqn:joint_moment} for any~$z_1,\dots,z_K \in U$. Moreover, the factor~$(-\i)^K$ is missing from the statement of the proposition, but it is present in the proof; see Equation (21) there.
\end{proof}

In addition, it is shown (\cite[Proposition 15]{Ken99}) that~$F_+$ is holomorphic in both variables, and~$F_-$ is anti-holomorphic in the first variable and holomorphic in the second variable. Moreover, the~$(1,0)$ forms~$F_+(z_1, z_2) d z_1$ and~$F_-(z_1, z_2) d \bar z_1$ are invariant under conformal maps. In other words, under changes of coordinates~$F_+$ transforms as a function in the second variable and as a holomorphic one form in the first variable, and~$F_-$ transforms as a function in the second variable and an anti-holomorphic one form in the first variable. Using~\eqref{eqn:joint_moment} in the case that each~$z_i \rightarrow A_{j_i}$ for some~$j_i$, it is deduced in~\cite{Ken99} that the joint centered height moments~$\mathbb{E}[Z_1^{n_1} \cdots Z_g^{n_g}] $ of heights~$h_1,\dots, h_g$ at the boundaries of the~$g$ holes are asymptotically invariant under conformal transformations; clearly, this conformal invariance holds more generally for any heights at~$z_1,\dots, z_K \in U$ as in~\eqref{eqn:joint_moment}.

\begin{figure}
    \centering
\includegraphics[width=0.33\linewidth]{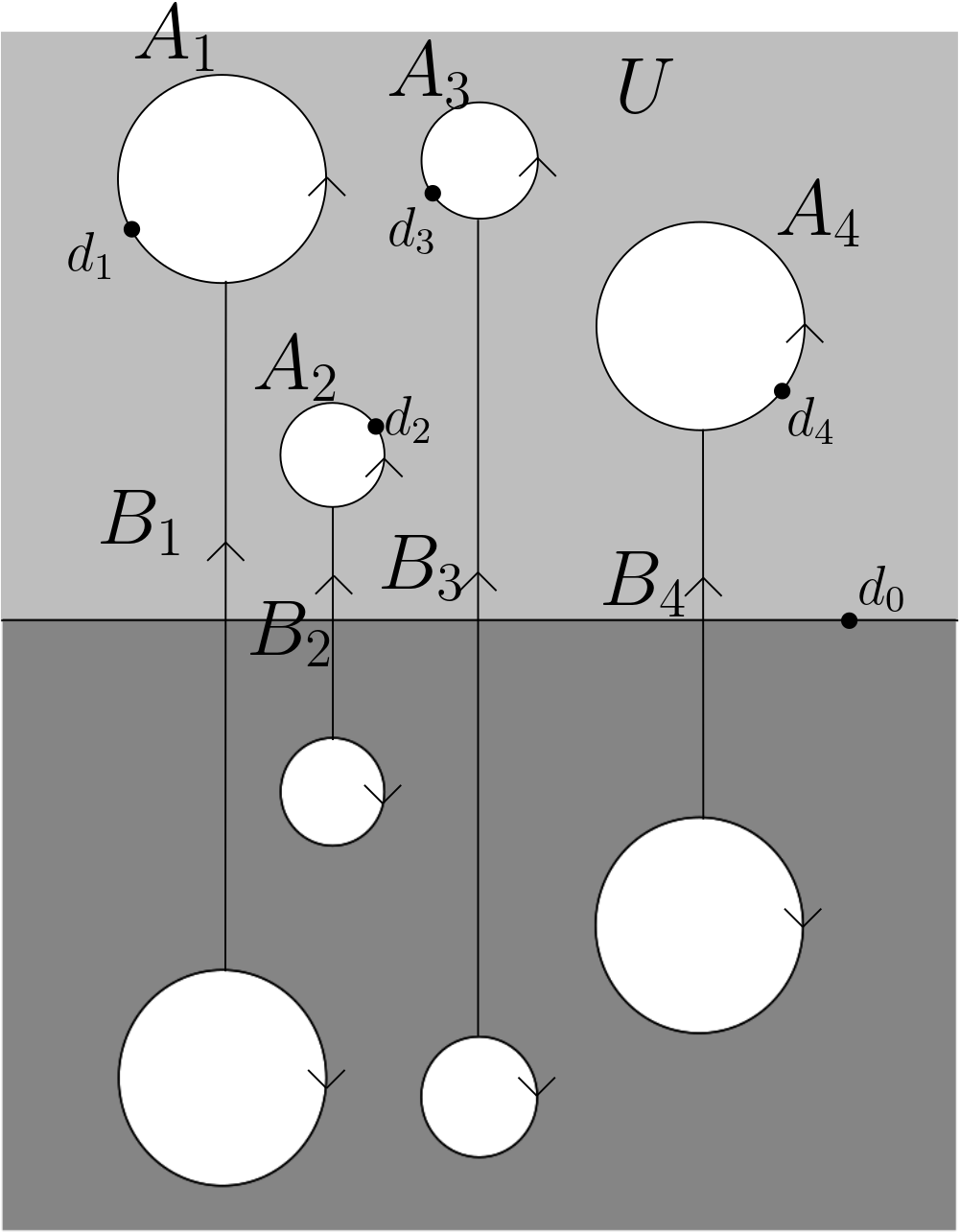}
    \caption{The~$M$ curve obtained as double of the planar domain~$U$ with marked points~$d_0,\dots,d_g$, in the case when~$U$ is the upper half plane with~$g$ circular holes cut out.}
    \label{fig:Udouble}
\end{figure}

\subsection{The double of a planar domain, theta functions, and prime forms}\label{subsec:doubleU}

 We will ultimately write a formula the functions~$F_+$,~$F_-$ in terms of theta functions defined on the compact Riemann surface~$R$, which we define as the double of~$U$. Though it is standard, to be concrete, we will explicitly describe the conformal structure on the double of~$U$. Then, we will briefly outline several facts about compact Riemann surfaces, and in particular about a class known as~$M$ curves into which our Riemann surface falls. We will also present the properties of theta functions and prime forms needed in our construction; the reader is referred to the wonderful reference~\cite{BK11} for more details, and also to the classical texts~\cite{Fay73} and~\cite{Mum07a}.

Define~$R$ as the surface obtained by gluing~$U$ to itself along its boundary. The natural map~$\sigma : R \rightarrow R$ given by swapping the copies of~$U$ will be antiholomorphic, once~$R$ is equipped with a conformal structure, and the fixed point set of~$\sigma$ is~$\partial U$, which we assume consists of~$g+1$ piecewise smooth boundary curves. So~$R = U \sqcup \sigma U / \text{gluing}$. Topologically,~$R$ is a compact genus~$g$ closed surface. Local charts for~$R$ can be defined as follows. For any neighborhood~$V$ contained in the interior of~$U$, use the natural coordinate~$z$ on~$U \subset \mathbb{C}$ as a local coordinate; for~$\sigma V$ (which is another copy of~$V$), use the coordinate~$\bar z$.  For a neighborhood~$V$ around a point in~$\partial U$ satisfying~$\sigma(V) = V$, as a local coordinate use a homeomorphism~$\phi$ which maps~$V$ to a symmetric-under-conjugation neighborhood in the upper half plane, such that~$V \cap \partial U$ maps into~$\mathbb{R}$, and~$V \cap U$ is conformally mapped to~$\phi(V) \cap \mathbb{H}$ (where~$\mathbb{H}$ is the upper half plane). We require that in such a local coordinate~$\phi$, the map~$\sigma$ corresponds to conjugation~$\phi \mapsto \bar \phi$. This provides~$R$ the structure of a Riemann surface, which we again emphasize is compact and has genus~$g$. Moreover, since the fixed point set of the antiholomorphic involution~$\sigma$ consists of~$g+1$ ovals, this surface is a so-called~\emph{M-curve}; for an informative exposition on M-curves we refer the reader to~\cite[Section 2]{BCT22}.

By the conformal invariance property discussed after Proposition~\ref{prop:moments}, before performing our analysis we may uniformize~$U$ to a certain model space (we do this for no reason other than for concreteness). The Koebe uniformization theorem implies that~$U$ can be conformally mapped to the upper half plane with~$g$ circular holes cut out, so from now on we assume that~$U$ is one such domain. With this realization of~$U$,~$R$ is given by gluing~$U$ to its conjugate along~$\mathbb{R}$, with conjugate pairs of circles identified. Away from the boundaries of circular holes, the coordinate~$z \in \mathbb{C}$ can be used for the surface, and~$\sigma(z) = \bar z$ will be the complex conjugation map. Throughout this note, we will talk about actual points, say,~$q_1, q_2$ on the surface in terms of their~$z$ coordinates~$z_1, z_2$.

We choose cycles~$A_j$,~$j=0,\dots,g$ and~$B_j$,~$j=1,\dots,g$ as in Figure~\ref{fig:Udouble}. Note~$A_i \circ B_j = \delta_{i j}$,~$i,j=1,\dots,g$, with~$\circ $ denoting the intersection pairing. Denote by~$\vec{\omega} = (\omega_1,\dots, \omega_g)$ the basis of~$g$ holomorphic one forms on~$R$ dual to this choice of~$A_1,\dots, A_g$ and~$B_1,\dots, B_g$, normalized so~$\int_{A_i}\omega_j = \delta_{i j}$. Let~$B$ be the corresponding period matrix defined by~$B_{i j} = \int_{B_i} \omega_j$.  The matrix~$B$ is symmetric and has positive definite imaginary part. Since~$R$ is a so-called~\emph{M-curve},~$B$ is purely imaginary~\cite[Lemma 11]{BCT22}.

 Define the \emph{theta function}, which is an entire map~$\theta : \mathbb{C}^g \rightarrow \mathbb{C}$, by 
\begin{equation}\label{eqn:theta_def}
\theta(z) = \theta(z; B)  \coloneqq \sum_{n \in \mathbb{Z}^g} e^{\i \pi  (n\cdot B n + 2 n\cdot z ) }.
\end{equation}
The theta function is quasi-periodic: It satisfies
\begin{equation}\label{eqn:theta_trans}
    \theta(z + m + B n) = \exp(-\i \pi n \cdot B n - 2 \i \pi n \cdot z) \theta(z) 
\end{equation}
for any~$m, n \in \mathbb{Z}^g$.

The~\emph{Jacobi variety} is defined as the quotient
\begin{equation}\label{eqn:jacvar}
J(R) \coloneqq \mathbb{C}^g / (\mathbb{Z}^g + B \mathbb{Z}^g) .
\end{equation}
Equation \eqref{eqn:theta_trans} states that~$\theta$ is quasi-periodic as a function on~$J(R)$.

 From~\eqref{eqn:theta_trans}, for a fixed~$e \in \mathbb{C}^g$, the function on the universal cover~$\widetilde{R} \rightarrow \mathbb{C}$ defined by
\begin{equation}\label{eqn:theta}
    z \mapsto \theta(\int_{d_0}^{z}\vec{\omega} + e)
\end{equation}
has a well defined set of zeros on~$R$; denote with~$D_e$ the formal sum of these zeros, or \emph{zero divisor}~$D_e = \sum_j p_j$. If the function~\eqref{eqn:theta} does not vanish identically, then~$D_e$ consists of~$g$ points (counted with multiplicity), and is uniquely determined by the property
\begin{equation}\label{eqn:theta_div}
\sum_{j=1}^g \int_{d_0}^{p_j}\vec{\omega} = -e + \Delta \qquad \text{ in } J(R)
\end{equation}
where~$\Delta \in J(R)$ is a special point called the \emph{vector of Riemann constants}.

We will also use the \emph{prime form}. Denoting $\tilde z_1$,~$\tilde z_2$ as lifts of~$z_1,z_2$ to the universal cover~$\tilde R$, the prime form is defined by
\begin{equation}\label{eqn:pf}
 E(z_1, z_2) =   \frac{\theta[f](\int_{\tilde z_2}^{\tilde z_1} \vec{\omega})}{\sqrt{H_f(z_1)} \sqrt{H_f(z_2)}}
\end{equation}
where~$f \in (\frac{1}{2} \mathbb{Z} / \mathbb{Z})^{2g}$ is any \emph{non-degenerate odd half-integer theta characteristic},~$\theta[f]$ is a \emph{theta function with characteristic}~$f$, which is a slightly modified version of the theta function, and~$H_f$ is a certain holomorphic one form on~$R$ which admits a well defined square root. We suppress dependence on choices of lifts in the left hand side of~\eqref{eqn:pf} because the expressions for height moments involving prime forms will be independent of the choices of lifts. The prime form~$E(z_1, z_2)$ is a~$(-\frac{1}{2},-\frac{1}{2})$ form on~$\tilde R \times \tilde R$, which means that in local coordinates (which we also call~$\tilde z_1, \tilde z_2$) 
\begin{equation}\label{eqn:pfloc}
E( z_1,  z_2) = \frac{c(\tilde z_1,\tilde z_2)}{\sqrt{d \tilde z_1} \sqrt{d \tilde z_2}}
\end{equation}
where the square roots in the denominator indicate (up to a sign) how~$E$ transforms under changes of variables. Two basic facts are that $E(z_1,z_2)$ does not depend on the choice of~$f$, and $E(z_1,z_2) = - E(z_2,z_1)$.  

In addition, the prime form satisfies the following properties for fixed~$z_1$:

\begin{enumerate}[(I)]
    \item It has a simple zero at any~$\tilde z_2$ such that~$ z_2 = z_1$, and no poles and no other zeros.\label{item:I}
    \item In local coordinates for~$\tilde z_2$ close to~$\tilde z_1$, we have~$E( z_1,   z_2) = \frac{\tilde z_2- \tilde z_1}{\sqrt{d \tilde z_1} \sqrt{d \tilde z_2}} + O(|\tilde z_1-\tilde z_2|^2)$. \label{item:II}
\item If~$z_2'$ is obtained by traversing the cycle~$A_j$ or~$B_i$ starting from~$z_2$, then~$E(z_1,z_2') = E(z_1,z_2)$ and 
     $E(z_1,z_2') = \exp( -\i \pi  B_{i i} - 2 \pi \i \int_{z_1}^{z_2}\omega_i) E(z_2,z_1)$, respectively. \label{item:III}
    
\end{enumerate}

\subsection{Extensions and properties of~$F_+$ and~$F_-$}

Now we would like to extend~$F_+$ and~$F_-$ to objects defined on~$R \times R$; towards this end, we first extend~$z_2 \mapsto F_0(z_1,z_2)$ and~$z_2 \mapsto F_1(z_1,z_2)$, so that we get maps defined on~$U \times R$. If~$z_2 \in \sigma U$, then let~$F_1( z_1,  z_2) = \overline{F_1(z_1, \bar z_2)} $. By the Schwarz reflection principle, this provides a holomorphic extension (away from~$z_1$ and~$d_1,\dots, d_g$) from~$U$ to all of~$R$ because~$\Im F_1$ vanishes for~$z_2 \in \partial U$. Similarly, define~$F_0( z_1,  z_2) = - \overline{F_0(z_1, \bar z_2)} $; this is an analytic extension because~$\Re F_0$ vanishes on~$\partial U$.

Next, we define~$F_+(z_1,z_2) = F_0(z_1,z_2) + F_1(z_1,z_2)$ and $F_-(z_1,z_2) = F_0(z_1,z_2) - F_1(z_1,z_2)$ as in~\eqref{eqn:fpdef} and~\eqref{eqn:fmdef}, where now~$z_2$ can vary over all of~$R$; however, note that so far~$F_{\pm}$ is only defined for~$z_1 \in U$. 

% We extend~$F_{\pm}(z_1, z_2)$ to all of~$R \times R$ by demanding
% \begin{equation}\label{eqn:last_extension}
% F_{+}(\bar z_1, z_2)  = -F_{-}(z_1,z_2).
% \end{equation}
% These extensions provide similar extensions to~$R \times R$ of~$F_0$ and~$F_1$, and on all of~$R\times R$ we continue to have~$F_{\pm} = F_0 \pm F_1$.

% \begin{remark}
%    The natural object, in the sense that it transforms covariantly under conformal maps, is~$F_+(z_1, z_2) d z_1$, which is a one form in~$z_1$. We would like to extend this one form to all of~$R$ by enforcing~$\sigma^* (F_+(z_1, z_2) d z_1) = \pm F_-(z_1, z_2) d \bar z_1$, for~$z_1 \in U$. We do this in a direct way: We compute exact formulas for~$F_{\pm}(z_1, z_2) d z_1$ when~$z_1$ is restricted to~$U$, and then observe that these can be used to define an extension of~$F_{\pm}(z_1, z_2) d z_1$ as~$z_1$ varies over all of~$R$.
% \end{remark}

Now we restate properties of~$F_+$ and~$F_-$ as a function of~$z_2 \in R$ for fixed~$z_1$, which follow from the discussion above together with the definitions of~$F_0$ and~$F_1$ given in Section~\ref{subsec:kenres}.
\begin{lem}\label{lem:F+properties}
    For any fixed~$z_1 \in U$, the function~$z_2 \mapsto F_+(z_1,z_2)$ from~$R \rightarrow \mathbb{C}$ satisfies the properties 
    \begin{enumerate}
        \item It is meromorphic in~$z_2$.
        \item It has a simple pole at~$z_2 = z_1$ with residue~$\frac{2}{\pi}$ and possibly a simple pole at~$d_1,\dots, d_g$, and has no other poles.
        \item It vanishes at~$z_2 = d_0$.
    \end{enumerate}
    
   %  Moreover,
   %  \begin{equation}\label{eqn:Fmmdef}
   % F_-(z_1,z_2) =  - F_+(\bar z_1, z_2).
   %  \end{equation}

   For any fixed~$z_1 \in U$, the function~$z_2 \mapsto F_-(z_1,z_2)$ from~$R \rightarrow \mathbb{C}$ satisfies the properties 
    \begin{enumerate}
        \item It is meromorphic in~$z_2$.
        \item It has a simple pole at~$z_2 = \bar z_1$ with residue~$-\frac{2}{\pi}$ and possibly a simple pole at~$d_1,\dots, d_g$, and has no other poles.
        \item It vanishes at~$z_2 = d_0$.
    \end{enumerate}
\end{lem}

\section{Computing the joint moments}
\label{sec:joint_moments}
\subsection{A formula for~$F_+$ and~$F_-$}
\label{subsec:FpFm}

We first have a lemma which gives an explicit representation of the functions~$F_+$ and~$F_-$, and it also extends their definitions to all of~$R \times R$. Define~$e \in J(R)$ by
\begin{equation}\label{eqn:edef}
    e \coloneqq -\sum_{j=1}^g \int_{d_0}^{d_j}\vec{\omega} + \Delta.
\end{equation}
By Lemma 19~\cite{BCT22} we have~$e \in \mathbb{R}^g/\mathbb{Z}^g$, and by Lemma 18 of the same work, the function~$z \mapsto \theta(\int_{d_0}^{z}\vec{\omega}+e)$ does not vanish identically, so \eqref{eqn:theta_div} describes its zero divisor. 
The identification of~$e$, which will be the shift in the theta functions used in the explicit expression below, is a crucial step in our computation.

\begin{lem}\label{lem:maincomp}
    We have, for all~$z_1 \in U$, \begin{equation}\label{eqn:Fpp}
   F_+(z_1,z_2) d z_1= \frac{2}{\pi} \frac{\theta(\int_{z_1}^{z_2} \vec{\omega} + e)}{\theta(e) E(z_1,z_2)} \frac{E(d_0,z_2) \theta(\int_{d_0}^{z_1} \vec{\omega} + e) }{E(d_0,z_1) \theta(\int_{d_0}^{z_2} \vec{\omega} + e)}
    \end{equation}
 and
 \begin{equation}\label{eqn:Fmm}
   F_-(z_1,z_2) d \bar z_1= - \sigma_{z_1}^* \left( F_+(z_1, z_2) d z_1 \right)
    \end{equation}
    where the right hand side denotes (minus) the pullback under~$\sigma$ of the one form~$F_+(z_1, z_2) d z_1$ in the variable~$z_1$.
\end{lem}
\begin{proof}
Temporarily denote the right hand side of~\eqref{eqn:Fpp} by~$\tilde{F}_{+}(z_1, z_2) dz_1$, for fixed~$z_1 \in U$. First, we observe (using the the quasi-periodicity properties~\eqref{eqn:theta_trans} and~\eqref{item:III} of the theta function and prime form) that the meromorphic function~$z_2 \mapsto \tilde{F}_{+}(z_1, z_2)$ picks up a monodromy factor of~$1$ around any cycle, i.e. it is well defined on~$R$. 

Next, observe
    \begin{equation}\label{eqn:rat}
      z_2 \mapsto   \frac{F_{+} (z_1,z_2)}{\tilde{F}_{+} (z_1,z_2)} 
    \end{equation}
    is holomorphic on~$R$ because~$z_2 \mapsto \tilde{F}_{+} (z_1,z_2)$ has a pole at~$z_1$, and at each~$d_j$,~$j=1,\dots,g$, due to the determining property~\eqref{eqn:theta_div} of the zero divisor of the theta function, which holds with~$d_j$ replacing~$p_j$ there if~$e$ is given by~\eqref{eqn:edef}; compare with the properties of~$F_+$ listed in Lemma~\ref{lem:F+properties}. Therefore,~\eqref{eqn:rat} must be constant ($R$ is compact). Sending~$z_2 \rightarrow z_1$ and using the behavior of the prime form at the diagonal, property~\eqref{item:II}, we see that the constant is~$1$, i.e.~$\tilde{F}_{+}(z_1,z_2) dz_1 = F_+(z_1,z_2) d z_1$.

    Clearly~\eqref{eqn:Fpp} can be extended to a meromorphic one form in~$z_1$ defined on all of~$R$, so that it is defeind on all of~$R \times R$. Using this, the right hand side of~\eqref{eqn:Fmm} makes sense, and similar arguments together with the second part of Lemma~\ref{lem:F+properties} can be used to prove its validity.
\end{proof}

\begin{figure}
    \centering
    \includegraphics[width=0.33\linewidth]{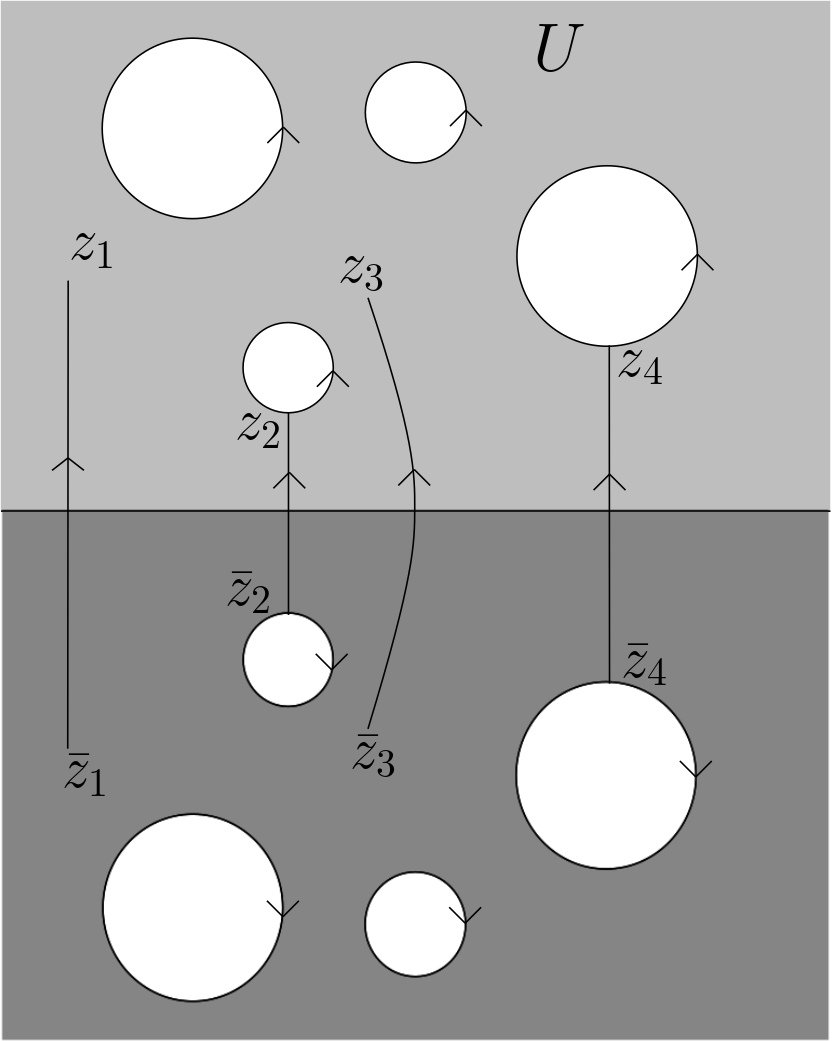}
    \caption{Integration paths for height moments. Some points~$z_i$ may be on boundary circles.}
    \label{fig:Udub_integrals}
\end{figure}

We now may rewrite the limiting joint height moment, the right hand side of the display in Proposition~\ref{prop:moments} above, in terms of integrals on the compact surface~$R$. Below when we write~$\int_{\bar z_1}^{z_1}\cdots \int_{\bar z_K}^{z_K}$ we mean integration over~$K$ disjoint paths in~$R$ connecting~$\bar z_i$ and~$z_i$, which are symmetric under conjugation, see Figure~\ref{fig:Udub_integrals}.
\begin{cor}\label{cor:moments2}
Define the~$(1/2,1/2)$ form on~$\widetilde{R} \times \widetilde{R}$ by
\begin{equation}\label{eqn:omegadef}
    \omega_0(z_1',z_2') = 4 \frac{\theta(\int_{z_1'}^{z_2'} \vec{\omega} + e)}{\theta(e) E(z_1',z_2')}.
\end{equation}
Then, for any pairwise distinct~$z_1,\dots,z_K \in U$ approximated by lattice points $z_1^{(\epsilon)},\dots,z_K^{(\epsilon)}$ (which may be on boundary circles, in which case~$h(z_j^{(\epsilon)}) - \mathbb{E}[h(z_j^{(\epsilon)})]= Z_{i_j}$ for some~$i_j = 1,\dots g$) we have
\begin{equation}\label{eqn:joint_moment2}
     \lim_{\epsilon \rightarrow 0} \mathbb{E}\left[  \prod_{j=1}^K (h(z_j^{(\epsilon)}) - \mathbb{E} [h(z_j^{(\epsilon)})]) \right] = \frac{1}{(2\pi \i)^K} \int_{\bar z_1}^{z_1} \cdots \int_{\bar z_K}^{z_K} \det((1-\delta_{i j})\omega_0(z_i', z_j'))_{i,j=1}^K .
   \end{equation}
\end{cor}

\begin{proof}
We use Lemma~\ref{lem:maincomp} and Proposition~\ref{prop:moments}.

We must ensure that the extra sign in the right hand side of~\eqref{eqn:Fmm} does not contribute. However in each summand in~\eqref{eqn:joint_moment}, for each term in the expansion of the determinant (as a sum over permutations) there are an even number of appearances of~$F_-$, so the signs cancel out.

Moreover, the factor of the form~$\frac{g(z_1)}{g(z_2)}$ on the right hand side of~\eqref{eqn:Fpp} cancels out in determinants, leading to the formula above.
\end{proof}

\subsection{Completing the proof}

Using the formula~\eqref{eqn:Fpp} in the expression~\eqref{eqn:joint_moment2}, we can now prove Theorems \ref{thm:intro_main1} and~\ref{thm:intro_main2}. Both theorems follows from results of~\cite{BN25}; for completeness, we outline the proof below, and refer the reader to Lemma 4.9, Proposition 4.8 together with Theorem 4.1, and Proposition 4.12 in that work for details in parts 1,2, and 3 of the proof below, respectively; for these computations, the surface~$\mathcal{R}$ there plays the role of~$R$ here, and~$\mathcal{R}_0$ there plays the role of~$U$ here.

Before beginning, we recall that the \emph{classical cumulant}~$\kappa[X_1,\dots, X_n]$ associated to a collection of random variables ~$X_{1},\dots, X_{n}$ (possibly with duplicates) is defined inductively by the relations
\begin{equation}\label{eqn:cumsum}
    \mathbb{E}[X_1 \cdots X_n] = \sum_{\pi} \prod_{B \in \pi} \kappa[X_i ; i \in B]
\end{equation}
where the summation is over partitions of indices~$\{1,\dots, n\}$ and the product is over blocks~$B$ in the partition~$\pi$.

\begin{proof}[proof of Theorem~\ref{thm:intro_main1}]
We will break up the proof outline into steps.
\begin{enumerate}
    \item We derive an expression for the joint cumulants of~$h(z_i^{(\epsilon)})$, and work with these rather than centered moments. The (limiting) cumulants have the form 
    \begin{equation}\label{eqn:cumeq}
    \kappa[h(z_i^{(\epsilon)}), i =1,\dots,K] = \frac{(-1)^{K+1}}{(2\pi  \i)^K}\int_{\bar z_1}^{z_1} \cdots \int_{\bar z_K}^{z_K} \sum_{ \text{$K$-cycles } \sigma} \prod_{j=1}^K \omega_0(z_j',z_{\sigma(j)}') + o(1).
    \end{equation}
    The sum is over permutations of~$\{1,\dots, K\}$ which consist of a single~$K$-cycle. This is boils down to a combinatorial fact ``under the integral'': The algebraic relationship between the sum over~$K$-cycles and the determinant is the same as the algebraic relationship~\eqref{eqn:cumsum} between cumulants and moments. Since any~$z_i$ may be on the boundary, this also gives joint cumulants between values of~$\tilde h (z_i^{(\epsilon)})$ and any collection of~$(Z_1,\dots,Z_g)$ (in this case the integrations corresponding to copies of~$Z_i$ will be over~$B$ cycles~$B_i$).

\item Recall harmonic functions~$f_i : U\rightarrow \mathbb{R}$,~$i=1,\dots,g$, from the Introduction. We may compute~$\lim_{\epsilon \rightarrow 0} \kappa[\tilde h(z_1^{(\epsilon)}), \cdots, \tilde h(z_K^{(\epsilon)})]$ by expanding out the product using multilinearity of cumulants together with~\eqref{eqn:cumeq}. The result is
\begin{equation}\label{eqn:joint_cum}
    \frac{(-1)^{K+1}}{(2\pi  \i)^K}\left(\int_{\bar z_1}^{z_1} -\sum_{j=1}^g f_j(z_1)\int_{B_j}\right)  \cdots \left(\int_{\bar z_K}^{z_K} -\sum_{j=1}^g f_j(z_k)\int_{B_j}\right)  \sum_{ \text{$K$-cycles } \sigma} \prod_{j=1}^K \omega_0(z_j',z_{\sigma(j)}')
\end{equation}
where the product of sums of integration symbols should be ``expanded out''.

We first analyze~\eqref{eqn:joint_cum} for~$K=2$, which is the second cumulant, or the second centered moment; this may be done verbatim as in \cite[Proposition 4.8]{BN25}. Expanding the expression into a sum of integrals, we see it is harmonic as a function of~$z_1$ and satisfies Dirichlet boundary conditions, and by analyzing the singularity as~$z_1 \rightarrow z_2$ (coming from the singularity of~$\omega_0$), we can see that it agrees with~$16/\pi$ times the Green's function~$g_U(z_1, z_2)$.

Then, we analyze higher cumulants, i.e.~\eqref{eqn:joint_cum} when~$K > 2$. When~$K>2$, the integrand is holomorphic in all variables, i.e. it has no poles. We can see this by observing that swapping~$z_1$ and~$z_2$ leaves the integrand invariant, which means that a simple pole (which is the only possible type of singularity) as~$z_1 \rightarrow z_2$ is impossible. Moreover, the expression vanishes as any variable~$z_i$ converges to~$\partial U$. Therefore, the higher cumulants are harmonic in~$z_1$ for any fixed distinct~$z_2,\dots, z_K$, and have zero boundary values, and thus vanish identically. The vanishing of higher cumulants implies the Wick rule for higher moments. This proves that~$\tilde{h}$ converges in the sense of moments to the Gaussian free field. 

\item To show that~$\tilde h$ and~$h_1,\dots, h_g$ are independent, we show
$$\kappa[\tilde h(z_1^{(\epsilon)}), \cdots, \tilde h(z_K^{(\epsilon)}), Z_1,\dots,Z_1, Z_2,\dots, \cdots, \dots, Z_g,\dots, Z_g] \rightarrow 0$$
where above there are any number~$n_i \geq 0$ of copies of each~$Z_i$. We must again analyze an expression like~\eqref{eqn:joint_cum} but now with, say,~$m=\sum_{i=1}^g n_i$ extra integrals over various~$B$ cycles. Similar arguments to the ones in the final paragraph of the last step lead to the vanishing of such a joint cumulant, which implies asymptotic independence (in the sense of moments). \qedhere
\end{enumerate}
\end{proof}
Next, we complete the proof of Theorem~\ref{thm:intro_main2}. We again give a very brief outline, since as we explain below, the proof consists of computations which can be taken word for word from~\cite{BN25}. 
\begin{proof}[Proof of Theorem~\ref{thm:intro_main2}]
We must show moments of~$(\frac{1}{4} Z_1,\dots, \frac{1}{4} Z_g)$ asymptotically match those of $(X_1-\mathbb{E}[X_1],\dots, X_g-\mathbb{E}[X_g])$, where~$(X_1,\dots, X_g)$ is a discrete Gaussian distribution as in the theorem statement. It suffices to match the joint cumulants of size~$\geq 2$. Denote~$\kappa_{n_1,\dots,n_g}$ as the (leading order asymptotic of the) joint cumulant of the collection of random variables consisting of~$n_1$ copies of~$\frac{1}{4}Z_1$,~$n_2$ copies of~$\frac{1}{4}Z_2$, and so on. To compute this, we take all variables~$z_i$ in~\eqref{eqn:cumeq} to the inner boundaries, so that all integrations are over~$B$ cycles, leading to the formula 
\begin{equation}\label{eqn:joint_Zcums}
\kappa_{n_1,\dots,n_g} = \frac{(-1)^{K+1}}{(2\pi \i)^K 4^K}\int_{B_1} \cdots \int_{B_1} \cdots \cdots \int_{B_g} \cdots \int_{B_g} \sum_{\text{$K$-cycles } \sigma} \prod_{j=1}^K  \omega_0(z_j',z_{\sigma(j)}')
\end{equation}
where there are~$n_i \geq 0$ integrations over the cycle~$B_i$.

The proof of Theorem 4.2 of~\cite{BN25} computes a formula in terms of theta functions for the expressions on the right hand side of~\eqref{eqn:joint_Zcums}, where~$\omega_0$ is of the form~\eqref{eqn:omegadef} for any~$e \in \mathbb{R}^g$  (up to the extra prefactor of~$4$, which is accounted for by the prefactor of~$\frac{1}{4^K}$). The theorem gives a formula for such expressions in terms of the theta function associated to~$R$: For~$K = n_1+\cdots+n_g \geq 2$, the right hand side of~\eqref{eqn:joint_Zcums} is given by 
\begin{equation}\label{eqn:theta_form}
    (2\pi \i)^K \kappa_{n_1,\dots, n_g} = \partial_{z_1}^{n_1} \cdots \partial_{z_g}^{n_g}\left( \log\theta(B z + e)  + \frac{1}{2}(2\pi \i) z\cdot B z \right)|_{z_1=\cdots=z_g = 0}.
\end{equation}
We outline the idea behind this computation. The proof is inductive. The~$K=2$ case can be computed directly using the identity in Equation (39) of~\cite{Fay73}. For the induction step, we will analyze the integrand. Making the~$e$ dependence explicit, denote 
$$
\Omega_K(z_1',\dots, z_K'; e) \coloneqq \frac{(-1)^{K+1}}{4^K} \sum_{\text{$K$-cycles } \sigma} \prod_{j=1}^K \omega_0(z_j',z_{\sigma(j)}').
$$
If $K \geq 2$, the quantities on the right hand side of~\eqref{eqn:theta_form} satisfy the property that passing from~$n_i \rightarrow n_i+1$ leads to another differentiation in~$z_i$ before setting~$z_1=\dots=z_g=0$, which is equivalent to applying the linear combination~$\sum_{j=1}^g B_{i j} \partial_{e_j}$ of derivatives in the variables~$(e_1,\dots,e_g)$. By induction, it suffices to show that the right hand side of~\eqref{eqn:joint_Zcums} satisfies the same property. The proof uses an identity of Fay (specifically, Equation (38) in Proposition 2.10) and some computations to show that~$\Omega_K(z_1',\dots, z_K'; e) = \sum_{i=1}^g \partial_{e_i} \Omega_{K-1}(z_2',\dots, z_K'; e) \omega_i(z_1')$, which implies the property we want for~\eqref{eqn:joint_Zcums}; recall~$\{\omega_i\}_{i=1}^g$ are a basis of holomorphic one forms, and they satisfy~$B_{ij} = \int_{B_j} \omega_i$.

Then, the \emph{modular transformation} implies that the these expressions~\eqref{eqn:theta_form} for limiting cumulants match the cumulants of a discrete Gaussian, see Corollary 4.15 in~\cite{BN25}. The resulting discrete Gaussian has the same distribution as in Theorem~\ref{thm:intro_main2}, in particular the shift~$e$ is the same, except the parameter~$\tau$ (defined by~\eqref{eqn:taudef}) is replaced by~$\i B^{-1}$. (We remark that in the notation of~\cite{BN25}, the \emph{scale matrix} corresponding to the distribution \eqref{eqn:disc_gauss} is instead defined to be~$\i \tau$, since the scale matrix there is normalized to be pure imaginary with positive definite imaginary part). However, the computations in Section 4.5 of~\cite{BN25}, especially Equation (98) there, imply that~$\i B^{-1} = \tau$, so the scale matrix also matches the one in Theorem~\ref{thm:intro_main2}. Finally, the discrete Gaussian is uniquely determined by its moments, so convergence of moments implies convergence in distribution.

\end{proof}

\bibliographystyle{alpha}
\bibliography{bibliotek}

\end{document}